\documentclass[11pt]{amsart}

\usepackage[american]{babel} 
\usepackage{latexsym, amssymb, amsmath, amsthm} 
\usepackage{tikz}
\usepackage{hyperref}

\hypersetup{
  colorlinks  = true, 
  urlcolor    = black, 
  linkcolor   = black, 
  citecolor   = black 
}

\numberwithin{equation}{section}

\newtheorem{theorem}{Theorem}[section]
\newtheorem{lemma}[theorem]{Lemma}
\newtheorem{proposition}[theorem]{Proposition}

\theoremstyle{definition}
\newtheorem{definition}[theorem]{Definition}

\newtheorem{example}[theorem]{Example}

\newcommand{\NN}{\mathbb{N}}
\newcommand{\ZZ}{\mathbb{Z}}

\newcommand{\PP}{\mathbb{P}}
\newcommand{\TT}{\mathbb{T}}

\newcommand{\C}{\mathcal{C}}

\newcommand{\supp}{\operatorname{supp}}

\newcommand{\reg}{\operatorname{reg}}

\renewcommand{\tt}{\mathbf{t}}
\newcommand{\J}{\mathcal{J}}
\newcommand{\K}{\mathcal{K}}
\newcommand{\T}{\mathcal{T}}
\newcommand{\E}{\mathcal{E}}
\newcommand{\B}{\mathcal{B}}

\newcommand{\ts}{\textstyle}
\renewcommand{\ss}{\scriptstyle}

\begin{document}

\title[Parameterized codes over graphs]{Parameterized codes over graphs}

\thanks{The first author was partially supported by the 
Centre for Mathematics of the University of Coimbra - UIDB/00324/2020, 
funded by the Portuguese Government through FCT/MCTES. 
The second author was partially supported by FCT/Portugal through CAMGSD, IST-ID ,
projects UIDB/04459/2020 and UIDP/04459/2020.}

\author{Jorge Neves}
\address{\emph{J.~Neves}: University of Coimbra, CMUC, Department of Mathematics, 3001-501 Coimbra, Portugal.}
\email{neves@mat.uc.pt}

\author{Maria Vaz Pinto}
\address{\emph{M.~Vaz Pinto}: University of Lisbon, IST, Department of Mathematics,
Avenida Rovisco Pais, 1049-001 Lisbon, Portugal.}
\email{vazpinto@math.tecnico.ulisboa.pt}


\maketitle
\centerline{\footnotesize\emph{Dedicated to Rafael Villarreal, on the occasion of his 70th birthday.}}

\begin{abstract}
In this article we review known results on parameterized linear codes over graphs, introduced
by Renter\'ia, Simis and Villarreal in 2011. Very little is known about their basic parameters and invariants. 
We review in detail the parameters dimension, regularity and minimum distance. As regards the parameter 
\emph{dimension}, we explore the connection to Eulerian ideals in the ternary case and we give new combinatorial formulas. 
\end{abstract}


\section{Introduction}\label{sec: intro}

A parameterized code over a graph is a linear code obtained by evaluating forms of 
fixed degree on a set of points obtained from the graph, in projective space over a finite field. 
They were introduced by Renter\'ia, Simis and Villarreal in \cite{ReSiVi} and, with some exceptions,
their study is wide open. In this article we will touch upon the basic parameters
and invariants of these codes, reviewing known results. Section~\ref{sec: dim} concerns the parameter \emph{dimension}
and focuses on the case of ternary linear codes, by exploring the relation with Eulerian ideals. 
Theorem~\ref{thm: characterizing standard monomials}, which gives a combinatorial formula for the dimension of 
parameterized code over a graph in the ternary case, and Theorem~\ref{thm: dimension even cycle}, which 
gives this formula explicitly in the case of an even cycle, are both new. 
Section~\ref{sec: reg} is dedicated to the invariant \emph{regularity} and Section~\ref{sec: mindist}
to the parameter \emph{minimum distance}.

\medskip

Let $G$ be a simple graph. We assume that $V_G=\{1,2,\dots,n\}$ and 
we denote $s=|E_G|$, which we always assume positive. We also 
fix a choice of ordering of the edges, $e_1,\dots,e_s$. 
Take $K$ to be a field and consider the two polynomial rings $K[x_1,\dots,x_n]$
and $K[t_1,\dots,t_s]$. (It is convenient to identify $E_G$ with the set $\{t_1,\dots,t_s\}$. 
Thus we may refer to the monomial obtained by multiplying a given set of edges.) 
Defining a homorphism of polynomial rings 
$\varphi \colon K[t_1,\dots,t_s] \to K[x_1,\dots,x_n]$ by
$$
t_k \mapsto x_ix_j 
$$
if and only if $t_k$ is the edge $\{i,j\}$, we obtain a rational 
map of $\PP^{n-1}$ to $\PP^{s-1}$, which, when restricted to the projective torus
$$
\TT^{n-1} = \{(x_1,\dots,x_n)\in \PP^{n-1} : x_i \not = 0, \text{ for all } i\},
$$
is a morphism. We denote the image of $\TT^{n-1}$ by this morphism by $X$. 
This set is then a subset (and, moreover, subgroup)
of the corresponding projective torus in $\PP^{s-1}$. The set $X$ is called the
\emph{projective algebraic toric set parameterized by the edges of $G$}.
Assume $K$ is finite. Then $X$ is also finite and 
the number of its elements can be determined as a function of $G$ (see Theorem~\ref{thm: number of points}, below). 
At this point, let us denote this number by $m$ and let $X=\{P_1,\dots,P_{m}\}$ correspond to a choice of 
ordering. Let $d\geq 0$. Then, the \emph{parameterized code of order $d$
over $G$}, denoted by $C_X(d) \subseteq K^{m}$,
is the image of the space of homogeneous polynomials in $t_1,\dots,t_s$, of degree $d$, by the map 
defined by 
\begin{equation}\label{eq: definition of ev}
f \to \left (\frac{f(P_1)}{f_0(P_1)},\dots,\frac{f(P_m)}{f_0(P_m)} \right)\in K^m,
\end{equation}
for every $f\in K[t_1,\dots,t_s]_d$ and where $f_0 = t_1^d$. 
\medskip 

A graph gives a sequence of linear codes:
$$
C_X(0), C_X(1),\dots, C_X(d),\dots
$$
all of which are subspaces of $K^m$. The list of dimensions of the codes in this sequence starts with $1$ and is
stricly increasing until it reaches $m$. (We will explain this in more 
detail in Section~\ref{sec: dim}). From a coding theory point of view, the degree at which the dimension of 
$C_X(d)$ reaches $m$ is an important parameter of this construction. 
We call it the \emph{index of regularity} (or, simply the regularity) for reasons we will explain later.
Other important invariants of the codes include their \emph{minimum distances}, which is the 
minimum number of nonzero components of a vector over all non-zero vectors in the code, 
and their \emph{length} (the number of components of a vector); which in this construction is $m$, common to all
codes in the sequence. Given that $C_X(d)$ are constructed from $G$, the expectation is that 
all of these invariants are in some way related to invariants of the graph. 
For a general graph, not much is known about the dimension and minimum distance of these codes. 
There has, however, been significant progress on the computation of the index of regularity and we 
will postpone a detailed account to Section~\ref{sec: reg}. As for the parameter length,
denoted above by $m=|X|$, a formula, holding for any graph, 
was given in \cite{even}. To state this result, let us denote the 
number of connected components of $G$ by $b_0(G)$ and let $q$ denote the cardinality of the field.

\begin{theorem}\label{thm: number of points} If $G$ is a bipartite graph then 
$$
|X|=(q-1)^{n-b_0(G)-1}.
$$
If $G$ is non-bipartite then 
$$
\renewcommand{\arraystretch}{1.3}
|X|= \left \{
\begin{array}{ll}
\bigl(\frac12\bigr)^{\gamma-1}(q-1)^{n-b_0(G)+\gamma-1} & \text{if $q$ is odd,}\\
(q-1)^{n-b_0(G)+\gamma-1} & \text{if $q$ is even,}
\end{array}
\right.
$$
where $\gamma$ is the number of 
non-bipartite components. 
\end{theorem}

\begin{proof}
See \cite[Theorem 3.2]{even}.
\end{proof}


\section{Dimension}\label{sec: dim}

From now on, let us denote $S=K[t_1,\dots,t_s]$ and let $I(X)\subseteq S$ be the homo\-ge\-ne\-ous vanishing ideal 
of $\{P_1,\dots,P_m\}$. Then $S/I(X)_d \simeq C_X(d)$ and therefore the dimension of $C_X(d)$, as 
$d\geq 0$, coincides with the Hilbert function of the module $S/I(X)$. 
Since $I(X)$ is the vanishing ideal of a set of points in projective space, 
we know that the Hilbert function of $S/I(X)$, and hence $\dim C_X(d)$, 
is strictly increasing until it reaches a constant value equal to the number of points of $X$. 
\medskip 

Denote the projective torus $\TT^{s-1}\subseteq \PP^{s-1}$ by $\TT$. As $X\subseteq \TT$ we get
\begin{equation}\label{eq: L9}
I(\TT) = (t_1^{q-1}-t_s^{q-1},\dots,t_{s-1}^{q-1}-t_s^{q-1})\subseteq I(X),
\end{equation}
From the point of view of the Hilbert Function,
the easiest case is when $X$ coincides with the projective torus $\TT = \TT^{s-1}\subseteq \PP^{s-1}$ and, hence, 
$I(X)$ is a complete intersection. We may use the Hilbert series of $S/I(\TT)$ to obtain
\begin{equation}\label{eq: dimensions for torus case}
\ts \dim C_\TT(d) = \sum_{j\geq 0} (-1)^j \binom{s-1}{j}\binom{s-1+d-(q-1)j}{s-1}.
\end{equation}
(see \cite{duursmaEtAl, secondhamming, torus} for details).  According to \cite[Theorem~4.4]{torus}, $X=\TT$ is the only case 
in which $I(X)$ is a complete intersection. Note that the formula of Theorem~\ref{thm: number of points}
gives $X=\TT$ if $G$ is a tree or, more generally, a forest, or when $G$ is a unicyclic graph 
with a unique odd cycle. On the opposite end of the class of
bipartite graphs are the complete bipartite graphs $\K_{a,b}$. In this case $I(X)$ is far from 
being a complete intersection, but the dimension 
function of $C_X(d)$ is known. To state it, let $k(s,d,q)$ be the summation on the right of \eqref{eq: dimensions for torus case}.
Then, 
$$
\ts \dim C_X(d) =  k(a,d,q)\, k(b,d,q).
$$
(see \cite[Theorem 5.2]{GonRen}). To our knowledge, these are the only two instances in which 
a formula for the dimension function of parameterized codes is known. 

\subsection{Dimension in the case of ternary codes}
When $K=\ZZ/3$, the situation is bettered by the recent results on the \emph{Eulerian ideal} of $G$. 
This ideal, defined in \cite{joinsAndEars}, is the pre-image of the ideal 
$$
(x_i^2-x_j^2 : 1\leq i,j\leq n)\subseteq K[x_1,\dots,x_n]
$$
by the map $\varphi$, defined at the begining of Section~\ref{sec: intro}. By \cite[Proposition~2.9]{joinsAndEars},
when $K=\ZZ/3$, the ideal $I(X)$ and the Eulerian ideal are the same. 
A set of generators which is, moreover, a Gr\"obner basis, is available from \cite{neves}. 
To state the result let us fix some notation. Given $\alpha = (\alpha_1,\dots,\alpha_s) \in \NN^s$
let us denote $t_1^{\alpha_1}\cdots t_s^{\alpha_s}$ by $\tt^{\alpha}$. We say that $\tt^\alpha-\tt^\beta$ is an Eulerian 
binomial if $\tt^\alpha$ and $\tt^\beta$ are relatively prime, square-free, of the same degree, and 
the edges with index set $\supp(\alpha)\sqcup \supp(\beta) \subseteq \{1,\dots,s\}$ induce 
a subgraph of $G$ with vertices of even degree; i.e., an Eulerian subgraph. We denote 
by $\E$ the (finite) set of all Eulerian binomials and by $\T= \{t^2_i-t^2_j : 1\leq i,j\leq s\}$.

\begin{theorem}\label{thm: Grobner basis}
Let $K=\ZZ/3$. The set of homogeneous binomials $\T\cup \E$ 
is a Gr\"obner basis of $I(X)$ with respect to the graded reverse lexicographic order in $S$.
\end{theorem}

\begin{proof}
See \cite[Theorem 3.3]{neves}. 
\end{proof}

In particular, $I(X)$ is generated in degree $\geq 2$. 
As 
$$
\dim C_X(d) = \dim_K (S/I(X))_d,
$$ 
we deduce that $\dim C_X(0) = 1$  and $\dim C_X(1) = s$, regardless of $G$. 
This holds also for any parameterized code over a graph, over \emph{any} finite field. 
\medskip

A technique that has always proved useful when trying to link the combinatorics of $G$ with 
the algebra of $S/I(X)$, is to take an Artinian quotient of this graded ring. 
This is specially easy to produce since any monomial in $S$ 
is $S/I(X)$-regular. (Indeed since $X$ is a subset of the projective torus $\TT\subseteq \PP^{s-1}$,
a monomial does not vanish at any point of $X$.) To study the dimension function of the 
codes the correct Artinian quotient is $S/(I(X),t_s^2)$, where $t_s$ 
is the last edge of the graph. 

\begin{definition}
Given $d\geq 0$, let $\B_d$ be the set of monomials of degree $d$ 
that are not divisible by any leading term of a polynomial in 
$(I(X),t_s^2)$, with respect to the graded reverse lexicographic order in $S$.
Extend the notation $\B_d$ to negative $d$ by setting $\B_d=\emptyset$ and 
denote the cardinality of $\B_d$ by $\beta(d)$.
\end{definition}

Since, for every $i=1,\dots,s$, $t_i^2$ is a leading term of an element of $(I(X),t_s^2)$
a monomial in $\B_d$ is necessarily square-free. In particular, $\B_d$ is surely empty as soon as 
$d>s$. As $(I(X),t_s^2)$ is generated in degrees $\geq 2$, we deduce that 
$\B_0=\{1\}$ and $\B_1=\{t_1,\dots,t_s\}$. As we show below, the elements of $\B_d$, 
correspond to special sets of edges of the graph. 
Before,  let us reveal the connection with the dimension function of the family of codes 
$C_X(d)$, $g\geq 0$.

\begin{proposition}\label{prop: dimension using standard monomials of Artinian reduction}
Let $K=\ZZ/3$ and $d\geq 0$. Then 
$$
\ts \dim C_X(d) = \sum_{i\geq 0} \beta(d-2i).
$$
\end{proposition}

\begin{proof}
Let us use induction on $d$.
It is clear that the formula holds for $d=0$ and $d=1$. 
Assume $d>1$. Since $t_s^2$ is $S/I(X)$-regular, the short exact sequence
\begin{equation}\label{eq: L94}
0\to S/I(X)[-2] \stackrel{\cdot t_s^2}{\longrightarrow  } S/I(X) \to S/(I(X),t_s^2) \to 0
\end{equation}
gives $\dim C_X(d) = \dim C_X(d-2) + \dim_K (S/(I(X),t_s^2))_d$. By Macaulay's Theorem, 
the cosets with representatives in $\B_d$ form a $K$-basis of the vector space 
of $(S/(I(X),t_s^2))_d$. In other words, $\beta(d) = \dim_K (S/(I(X),t_s^2))_d$. Hence the formula follows by induction. 
\end{proof}

The key to get a combinatorial formula for $\dim C_X(d)$ is then the combinatorial characterization of 
the elements of $\B_d$. For this, we need a Gr\"obner basis of $(I(X),t_s^2)$, which is easily obtained 
from that of $I(X)$.

\begin{proposition}\label{prop: Grobner basis extended}
Let $K=\ZZ/3$. The set $\T\cup \E \cup \{t_s^2\}$
is a Gr\"obner basis of $(I(X),t^2_s)$ with respect to the graded reverse lexicographic order in $S$.
\end{proposition}

\begin{proof}
Since $t_s^2$ and the leading term of any binomial in $\T\cup \E$ are coprime, 
their $S$-polynomial reduces to zero. Since $\T \cup \E$ is a Gr\"obner basis,
the $S$-polynomials of all pairs of elements of $\T\cup \E$ also reduce to zero. 
\end{proof}

Let us now introduce the combinatorics. 

\begin{definition}[{\cite[Definition~4.4]{neves}}]\label{def: parity join}
$J\subseteq E_G$ is called a parity join if and only if $|J \cap  E_C | \leq \frac{|E_C|}{2}$,
for every Eulerian subgraph of $C\subset G$ with an even number of edges.
\end{definition}

The terminology of parity join comes from the relation with $T$-joins of cardinality of fixed 
parity, as explained in \cite{neves}. A parity join need not use half the edges of every 
Eulerian subgraph. When it does use half the edges of a given Eulerian subgraph, these need not include 
the last edge. 
\begin{definition}
Given $d\geq 0$, let $\J_d$ denote the set of parity joins, $J\subseteq E_G$, of cardinality 
$d$, that contain the last edge of every Eulerian subgraph $C\subseteq G$ for which 
$|J\cap E_C| = \frac{|E_C|}{2}$. Let us also extend this notation by 
setting $\J_d = \emptyset$, for all $d<0$.
\end{definition}

The proof of the next result is an adaptation of the ideas of \cite{neves}. There, the approach privileges 
fixed parity $T$-joins.

\begin{theorem}\label{thm: characterizing standard monomials}
Let $K=\ZZ/3$. The map $\B_d \to \J_d$ given by 
$$
\tt^\gamma \mapsto \{e_i : i\in \supp(\gamma)\}
$$ 
is well-defined and a bijection. In particular,
$$
\ts \dim C_X(d) = \sum_{i\geq 0} |\J_{d-2i}|.
$$
\end{theorem}

\begin{proof}

\noindent
As $\tt^\gamma \in \B_d$ is square-free, $\{e_i : i\in \supp(\gamma)\}$ is a set of $d$ edges.
Let $C\subseteq G$ be any Eulerian subgraph with an even number of edges. 
Assume
$$
\textstyle |\J(\tt^\gamma) \cap E_C|> \frac{|E_C|}{2}.
$$ 
Let $\tt^\alpha$ be the product of the first $\frac{|E_C|}{2}$ edges in $\J(\tt^\gamma)\cap E_C$ and 
let $\tt^\beta$ be the product of the remaining edges of $C$. Then $\tt^\alpha - \tt^\beta$
is an Eulerian binomial and, as $\tt^\beta$ is divisible by the last edge of $\J(\tt^\gamma)\cap E_C$, its
leading term is $\tt^\alpha$. But then $\tt^\alpha$ divides $\tt^\gamma \in \B_d$, and this is a contradiction. 
Hence 
$$
\textstyle |\J(\tt^\gamma) \cap E_C|\leq \frac{|E_C|}{2}. 
$$
We deduce that $\{e_i : i\in \supp(\gamma)\}$
is a parity join. Additionally, if 
$$
\ts |\J(\tt^\gamma) \cap E_C|= \frac{|E_C|}{2}
$$ but $\{e_i : i\in \supp(\gamma)\}$
does not contain the last edge of $C$, the same argument leads to a contradiction. 
Hence the map is well-defined. 
\medskip 

\noindent
It is clearly an injective map.
To prove surjectivity, let $J\in \J_d$, let $\tt^\gamma$ be the product of the 
edges in $J$ and let us show that $\tt^\gamma \in \B_d$.
Clearly $\deg(\tt^\gamma) = |J| = d$, so that all we need to show is that $\tt^\gamma$
is not divisible by any leading term of $(I(X),t_s^2)$. Since $\T\cup \E \cup \{ t_s^2 \}$
is a Gr\"obner basis for this ideal (Proposition~\ref{prop: Grobner basis extended}) it is enough 
to check that $\tt^\gamma$ is not divisible by the leading term of any element of $\T\cup \E \cup \{ t_s^2 \}$.
Since $\tt^\gamma$ is square-free, $t_i^2\nmid \tt \gamma$, for all 
$i=1,\dots,s$. Let $g = \tt^\alpha - \tt^\beta\in \E$, with $\operatorname{lt}(g)=\tt^\alpha$ 
(without loss of generality). Let $C\subseteq G$ be the corresponding Eulerian subgraph, i.e.,
the graph induced by $\{e_i : i \in \supp(\alpha)\} \sqcup \{e_j : j \in \supp(\beta)\}\subseteq E_G$
With a view to a contradiction, suppose that $\tt^\alpha \mid \tt^\gamma$.
Then, as $J$ is a parity join, 
$$
\ts |J\cap E_C| = \frac{|E_C|}{2}
$$
which implies that $J\cap E_C = \{e_i : i \in \supp(\alpha)\}$.
But if $J\in \J_d$ then $J$ must contain the last edge of $C$ which means that 
$\tt^\alpha$ is divisible by this edge. But this is a contradiction 
since we are assuming that $\operatorname{lt}(g)=\tt^\alpha$. Hence 
$\tt^\alpha \nmid \tt^\gamma$, for the leading term of any element of $\E$. 
We conclude that $\tt^\gamma\in \B_d$ and hence the map is also surjective.
This bijection yields $|\B_d| = |\J_d|$ and the formula 
for $\dim C_X(d)$ follows from Proposition~\ref{prop: dimension using standard monomials of Artinian reduction}.
\end{proof}

Let us illustrate the applications of this result by 
considering the case when $G$ has no Eulerian subgraphs with an even number of edges.
Note that, by Theorem~\ref{thm: Grobner basis}, $\E = \emptyset$ so that  
$$
I(X)= (\T) = (t_1^2-t_s^2,\dots,t_{s-1}^2-t_s^2)
$$
is a complete intersection and the dimension of $C_X(d)$ is given by \eqref{eq: dimensions for torus case}, with $q=3$.
If $G$ possesses no Eulerian subgraphs with even number of edges then every 
subset of edges is a parity join, hence 
$$
\J_d = \{J\subseteq E_G : |J|= d\}.
$$
Then, by Theorem~\ref{thm: characterizing standard monomials}, 
$$
\ts \dim C_X(d) = \sum_{i\geq 0}^{k}\binom{s}{d-2i}.
$$
To see that this amounts to the same as \eqref{eq: dimensions for torus case} with $q=3$, let us manipulate 
the Hilbert series of $S/I(X)$, as in \cite{torus}, but aiming at our formula. Since the ideal 
$I(X)\subseteq K[t_1,\dots,t_s]$ is a complete intersection of $s-1$ forms of degree two, the Hilbert series of $S/I(X)$ is 
$$
\frac{(1-T^2)^{s-1}}{(1-T)^s} = \frac{(1+T)^s}{1-T^2} = (1+T)^s\sum_{i\geq 0} T^{2i}.
$$
Equating the coefficient of $T^d$,
$$
\ts \dim C_X(d) = \dim_K S/I(X) = \sum_{i\geq 0} \binom{s}{d-2i}.
$$

We end this section by applying Theorem~\ref{thm: characterizing standard monomials} 
to the case of an even cycle.  

\begin{theorem}\label{thm: dimension even cycle}
Let $K=\ZZ/3$ and let $G=\C_{2\ell}$ be a cycle of length $s=2\ell$. 
Then 
$$
\dim C_X(d) = 
\left \{
\begin{array}{ll}
2^{s-2}, & \text{if }d\geq \ell-1, \\
\sum_{i\geq 0} \binom{s}{d-2i}, & \text{if }0\leq d\leq \ell-2. 
\end{array}
\right.
$$
\end{theorem}

\begin{proof}
Given that a parity join in $G$ is simply a subset of $d\leq \ell$ edges, we get 
$\J_{\ell+i} = \emptyset$, for all $i>0$. Also, an element in $\J_\ell$ must contain the edge 
$t_s$ and so $|\J_\ell| = \binom{s-1}{\ell-1}$. For $0\leq d \leq \ell -1$, the elements of $\J_d$
are the sets of $d$ edges of $G$, without any condition. Thus $|\J_d| = \binom{s}{d}$. Using 
Theorem~\ref{thm: characterizing standard monomials}, if $0\leq d\leq \ell-1$,
$$
\ts \dim C_X(d) = \sum_{i\geq 0} |\J_{d-2i}| = \sum_{i\geq 0} \binom{s}{d-2i}.
$$
The sum of all binomial coefficients of lower indices of the same parity is well-known:
$$
\ts \sum_{i\geq 0} \binom{s}{\ell-1-2i} + \sum_{i\geq 0} \binom{s}{\ell+1+2i} = 2^{s-1}.
$$
Since $s=2\ell$ and hence $\binom{s}{\ell-1-2i} = \binom{s}{\ell+1+2i}$
we deduce that $\dim C_X(\ell-1) = 2^{s-2}$. If $d = \ell$, using Pascal's identity and the same kind of argument as above,
$$
\renewcommand{\arraystretch}{1.5}
\begin{array}{l}
\dim C_X(\ell) = \sum_{i\geq 1} \binom{s}{\ell-2i} + \binom{s-1}{\ell-1}\\
\phantom{\dim C_X(\ell) } = \sum_{i\geq 1} \binom{s-1}{\ell-1-2i} + \sum_{i\geq 1} \binom{s-1}{\ell-2i} + \binom{s-1}{\ell-1} \\
\phantom{\dim C_X(\ell) } = \sum_{i\geq 1} \binom{s-1}{\ell-1-2i} + \binom{s-1}{\ell-1} + \sum_{i\geq 1} \binom{s-1}{\ell-1+2i}\\
\phantom{\dim C_X(\ell) } = 2^{s-2}.
\end{array}
$$
Finally, if $d>\ell$, given that $|\J_{\ell +i}|=0$, for all $i>0$ and given 
the formula of Theorem~\ref{thm: characterizing standard monomials},
we deduce that $\dim C_X(d)$ is equal to either $\sum_{i\geq 0} |\J_{\ell-2i}|$ or to $\sum_{i\geq 0} |\J_{\ell-1-2i}|$,
both of which are equal to $2^{s-2}$.
\end{proof}


\section{Regularity}\label{sec: reg}

Since $I(X)$ is the vanishing
ideal of a set of $m$ points in projective space, the Hilbert polynomial of $S/I(X)$ is constant and equal to 
$m$. In other words, there exists $r$ such that 
$$
\dim C_d(X) = m\iff C_d(X) = K^m,
$$
for all $d\geq r$. (From the coding theory point of view, this is where 
$C_d(X)$ becomes a trivial linear code.) The least $r$ in these conditions is 
called the \emph{index of regularity} of $S/I(X)$. Since any monomial is $S/I(X)$
regular, this module is $1$-dimensional and Cohen--Macaulay. Hence  
the index of regularity coincides with the Castelnuovo--Mumford regularity of 
$S/I(X)$. From now on we will refer to this integer simply by the \emph{regularity} of $S/I(X)$ and 
we will denote it by $\reg S/I(X)$. The next table summarizes the early known results regarding this invariant.
\medskip

\begin{table}[h]
\renewcommand{\arraystretch}{1.5}
\begin{center}
\begin{tabular}{l|l}
 & $\reg S/I(X)$ \\
\hline
$X=\TT^{s-1}$ & $(s-1)(q-2)$ \\
\hline
$G=\mathcal{K}_{a,b}$ & $(\max\{a,b\}-1)(q-2)$ \\
\hline
$G=\mathcal{K}_n, \:\: \scriptstyle n>3$ & $\lceil (n-1)(q-2)/2 \rceil$ \\
\hline
$G=\C_{2\ell}$ & $(\ell-1)(q-2)$ \\
\hline
$G=\mathcal{K}_{a_1,\dots,a_r}, \:\: \scriptstyle r>2$ & $\max\{a_1(q-2),\dots,a_r(q-2),\lceil (n-1)(q-2)/2\rceil\}$ \\
\hline
\end{tabular}
\end{center}
\medskip
\caption{Known values of $\reg S/I(X)$}
\label{table: values of reg}
\end{table}

In Table~\ref{table: values of reg}, $\K_n$ denotes a complete graph on $n>3$ vertices. The value for 
the regularity was given in \cite[Remark~3]{GoReSa13}. In the case of the complete bipartite graph, the regularity 
was obtained in \cite[Corollary~5.4]{GonRen} and the case of
an even cycle, $G=\C_{2\ell}$, in \cite[Theorem~6.2]{even}. The value of regularity for a complete multipartite graph
on $n=a_1+\cdots+a_r$ vertices,
denoted here by \mbox{$G=\mathcal{K}_{a_1,\dots,a_r}$}, 
was given in \cite[Theorem~4.3]{multipartite}.

\subsection{Parallel compositions} A graph is a parallel composition of paths if there exist path graphs $P_1,P_2,\dots,P_r$
such that $G$ is obained by identifying all the first end-points of the paths into a single vertex and all of the second 
end-points of the paths into another vertex. We have used \emph{first} and \emph{second} for the sake of clarity; we do not fix any orientation on the paths. Figure~\ref{fig: parallel composition} illstrates this definiton. 
\begin{figure}[h]
\begin{center}
\begin{tikzpicture}[line cap=round,line join=round, scale=1.5]

\draw [fill=black] (0,-.6) circle (1pt);
\draw [fill=black] (3.45,-.6) circle (1pt);

\draw (0,-.6)..controls (0.15,0)..(.75,0);
\draw [fill=black] \foreach \x in {0.75,1.5}
    {(\x,0)  circle (1pt) -- (\x+.75,0)  };
\draw (2.5,0) node{$\cdots$};
\draw (2.7,0)..controls (3.25,0)..(3.45,-.6);
\draw (3,.15) node {$P_1$};

\draw (0,-.6)..controls (0.15,-.45)..(.75,-.45);
\draw [fill=black] \foreach \x in {0.75,1.5}
    {(\x,-.45)  circle (1pt) -- (\x+.75,-.45) };
\draw (2.5,-.45) node{$\cdots$};
\draw (2.7,-.45)..controls (3.25,-.45)..(3.45,-.6);
\draw (3,-.3) node {$P_2$};

\draw (0,-.6)..controls (0.15,-1.2)..(.75,-1.2);
\draw [fill=black] \foreach \x in {0.75,1.5}
    {(\x,-1.2)  circle (1pt) -- (\x+.75,-1.2) };
\draw (2.5,-1.2) node{$\cdots$};
\draw (2.7,-1.2)..controls (3.25,-1.2)..(3.45,-.6);
\draw (3,-.67) node{$\vdots$};
\draw (3,-1.05) node {$P_r$};

\end{tikzpicture}
\end{center}
\caption{The parallel composition of paths $P_1,P_2,\dots,P_r$.}
\label{fig: parallel composition}
\end{figure}
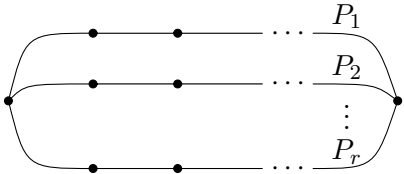
A parallel composition of paths may be bipartite or non-bipartite. 
The bipartite case is when the lengths of $P_i$ have the same parity. 
The value of the regularity of $S/I(X)$ for a graph of this type was computed in 
\cite{parallel}.

\begin{theorem}\label{thm: parallel}
Let $G$ be parallel composition of paths of lengths $k_1,\dots,k_r$, with $r\geq 2$. If 
$G$ is bipartite then 
$$
\reg S/I(X) =
\renewcommand{\arraystretch}{1.3}
\left \{
\begin{array}{ll}
 (\lfloor \frac{k_1}{2}\rfloor +\cdots +\lfloor \frac{k_r}{2} \rfloor)(q-2),& \text{if $k_i$ are odd,}\\
 ( \frac{k_1}{2} +\cdots + \frac{k_r}{2} -1)(q-2),&  \text{if $k_i$ are even.}
\end{array}
\right.
$$
If $G$ is non-bipartite then, assuming without loss of generality that $k_1,\dots,k_\ell$ are even and 
$k_{\ell+1},\dots,k_r$ are odd,  
$$
\renewcommand{\arraystretch}{1.3}
\reg S/I(X) =
\left \{
\begin{array}{ll}
 (k_1+k_2-1)(q-2),&  \text{if\; $\ss \ell =1, r=2$,}\\
 (k_1+\lfloor \frac{k_2}{2}\rfloor +\cdots +\lfloor \frac{k_r}{2} \rfloor)(q-2),&  \text{if\; $\ss \ell =1, r>2$,}\\
 (\frac{k_1}{2} +\cdots + \frac{k_\ell}{2} + k_{\ell +1})(q-2),&  \text{if\; $\ss \ell >1, r=\ell +1$,}\\
 (\frac{k_1}{2} +\cdots + \frac{k_\ell}{2} + \lfloor \frac{k_{\ell +1}}{2}\rfloor +\cdots +\lfloor \frac{k_r}{2} \rfloor)(q-2),&  \text{if\; $\ss \ell >1, r>\ell +1$.}
\end{array}
\right.
$$
\end{theorem}

\subsection{Nested ear decompositions}
We say that $G$ is endowed with an open ear decomposition if there exist 
subgraphs $E_1,\dots,E_r$, with $E_1$ a cycle and $E_2,\dots,E_r$ 
paths such that, for each $i=2,\dots,r$, the end-points of $E_i$ 
are distinct and belong to $E_1\cup \cdots \cup E_{r-1}$, while 
all other vertices do not. The subgraphs $E_1,\dots,E_r$ are called 
the \emph{ears} of the decomposition. Given $i=2,\dots,r$, 
we say that $E_i$ determines a nest interval if both its end-points 
belong to the same $E_j$, for some $j<i$ and, in this case, 
we define the corresponding nest interval to be the sub-path of $E_j$ 
determined by the two end-points of $E_i$. (If $j=1$, we take any of the two 
sub-paths.) In \cite{eppstein}, Eppstein 
defines the notion of nested ear decomposition by requiring that, in addition 
to the original assumptions, all $E_i$, for $i=2,\dots,r$ determine a 
nest interval and, for any two nest intervals contained in a same ear $E_j$, 
either they are disjoint or one is contained in the other.

\begin{theorem}[{\cite[Theorem 4.4]{nested}}]\label{thm: regularity of nested ear decomposition}
Assume $G$ is bipartite and that $E_1,\dots,E_r$ is 
a nested ear decomposition of $G$ with $\epsilon$ ears of even length. Then 
$$
\ts \reg S/I(X) = \frac{|V_G|+\epsilon -3}{2}(q-2).
$$
\end{theorem}
\medskip 

Note that, in particular, it follows that the number of even length ears in any nested ear decomposition of a graph is constant. In the proof \cite[Theorem 4.4]{nested}, it is necessary to relax the definition of nested ear decomposition and, as a result, this theorem holds for a more general notion of ear decomposition called \emph{weak
nested ear decomposition}.

Any parallel composition of paths $P_1,\dots,P_r$
is endowed with a nested ear decomposition, simply by setting 
$E_1$ equal to $P_1\cup P_2$ and,
if $r>2$, by setting $E_i = P_{i+1}$, for all $i=2,\dots,r-1$. 
If the lengths of $P_i$ are all even, then $\epsilon$, with 
respect to the ear decomposition we have defined, is equal to $r-1$. 
As 
$$
\ts |V_G| = (\sum_{i=1}^r k_i) - r + 2
$$ 
we get:
$$
\reg S/I(X) = \ts \frac{|V_G|+\epsilon -3}{2}(q-2) = ( \frac{k_1}{2} +\cdots + \frac{k_r}{2} -1)(q-2),
$$
which agrees with Theorem~\ref{thm: parallel}. If the lengths of the paths are all 
odd, the same can be verified.
\medskip

\subsection{Regularity in the case of ternary codes} If $K=\ZZ/3$ then, as 
mentioned above, the vanishing ideal $I(X)$ coincides with the Eulerian ideal defined over $\ZZ/3$.  

\begin{theorem}[{\cite[Theorem~4.13]{neves}}]
\label{thm: regularity of Eulerian ideals}
Let $K=\ZZ/3$ and $G$ be any graph. Then $\reg S/I(X)$
is equal to the maximum cardinality of a parity join minus $1$.
\end{theorem}

We end this section with a purely combinatorial result on the maximal cardinality of a parity join, which is straightforward 
by combining the previous theorem with the formulas for the regularity given before, with $q=3$.

\begin{proposition}
Denote by $\K_n$ a complete graph on $n$ vertices, 
$\K_{a,b}$ a complete bipartite graph on $n=a+b$ vertices, 
$\K_{a_1,\dots,a_r}$ a complete multipartite graph on 
$n=a_1+\cdots+a_r$ vertices, where $r>2$, $\operatorname{Pc}(k_1,\dots,k_r)$ the parallel 
composition of $r$ paths of lengths $k_1,\dots,k_r$, 
and denote by $\mu(G)$ the maximal cardinality of a parity join. Let $H$ be any bipartite graph 
with a nested ear decomposition having $\epsilon$ even length ears. The following holds:
$$
\renewcommand{\arraystretch}{1.5}
\begin{tabular}{l|l}
 & $\mu(G)$ \\
\hline
$G=\mathcal{K}_{a,b}$ & $\max\{a,b\}$; \\
\hline
$G=\mathcal{K}_n, \:\: \scriptstyle n>3$ & $\lceil \frac{n-1}{2} \rceil +1$; \\
\hline
$G=\mathcal{K}_{\alpha_1,\dots,\alpha_r}$ & $\max\{\alpha_1,\dots,\alpha_r,\lceil \frac{n-1}{2}\rceil\} + 1$; \\
\hline
$G=\operatorname{Pc}(k_1,\dots,k_r)$ \text{and} $k_i$ \text{even} & $\frac{k_1}{2} +\cdots + \frac{k_r}{2}$;\\
\hline
$G=\operatorname{Pc}(k_1,\dots,k_r)$ \text{and} $k_i$ \text{odd} & $\lfloor \frac{k_1}{2}\rfloor +\cdots +\lfloor \frac{k_r}{2} \rfloor + 1$;\\
\hline
$G = H$ & $\frac{|V_G|+\epsilon -1}{2}$.\\
\hline
\end{tabular}
$$
\end{proposition}


\section{Minimum Distance}\label{sec: mindist}

We recall that the minimum distance $\delta_X(d)$ of the code $C_X(d) \subseteq K^m$ is defined as follows
$$
\delta_X(d)=\mbox{min}\{\,\|a\|, \,a=(a_1,\ldots,a_m) \in C_X(d), \,a \neq 0\,\},
$$
where $\|a\| = |\{i: a_i \neq 0\}|$. Clearly $1\leq \delta_d \leq m$.
The Singleton Bound (see \cite{stich}, p.41) tells us that 
$$\delta_X(d) \leq |X| - \mbox{dim }C_X(d) + 1.$$
Since for $d \geq$ reg $S/I(X)$, dim $C_X(d) = |X|$, we have $\delta_X(d) = 1$, 
for $d \geq$ reg $S/I(X)$.
Moreover, the minimum distance is strictly decreasing until it reaches 
$1$ (\cite{ReSiVi}, \cite{tohaneanu}):
$$
\left\{
\begin{array}{ccl}
\delta_X(d) > 1 & \Rightarrow & \,\delta_{X}(d) > \delta_X(d+1)\\
\delta_X(d) = 1 & \Rightarrow & \,\delta_X(d+1) = 1\\\end{array}
\right.
.$$

The minimum distance is a very difficult parameter to calculate. In the case of evaluation codes, 
this calculation corresponds to counting zeros of homogeneous polynomials.
The next theorem is one of the few cases where we have an explicit formula for the minimum distance.

\begin{theorem}{\cite[Theorem~3.4]{torus}}\label{min dist torus}
When $X=\TT$, the projective torus in $\mathbb{P}^{s-1}$, and $d\geq 1$, 
the minimum distance of $C_X(d)$ is given by 
$$
\delta_X(d)=\left\{\begin{array}{cll}
(q-1)^{s-(k+2)}(q-1-\ell)&\mbox{if}&d\leq (q-2)(s-1)-1\\
1&\mbox{if}&d\geq (q-2)(s-1)
\end{array}
 \right.
$$
where $k$ and $\ell$ are the unique integers such that $k\geq 0$,
$1\leq \ell\leq q-2$ and $d=k(q-2)+\ell$. 
\end{theorem}

Recall that we say that a linear code is maximum distance separable (MDS) if equality holds
in the Singleton Bound. By the theorem above (see also \cite{secondhamming}),
if $X=\TT$ is the projective torus in $\mathbb{P}^{1}$ and $d\geq 1$, 
then $C_X(d)$ is an MDS code and its minimum distance is given by 
$$
\delta_X(d)=\left\{\begin{array}{cll}
q-1-d&\mbox{if}&d\leq q-3\\
1&\mbox{if}&d\geq q-2
\end{array}
 \right.
.$$

As we have seen in Section~\ref{sec: dim}, if $G$ is a connected graph and $X$ is the projective 
algebraic toric set parameterized by the edges of $G$, then $X = \TT$ if and only if $G$ is a tree 
or $G$ is a unicyclic graph with a unique odd cycle.
If $G$ is a forest, we also have $X = \TT$.
Hence, for these graphs, $\delta_X(d)$ is known.

In the case $G$ is a complete bipartite graph, the minimum distance of $C_X(d)$ is also known;
it can be obtained from Theorem~\ref{min dist torus}  together with the following result:

\begin{theorem}{\cite[Theorem~5.5]{GonRen}}\label{mindist completebipartite}
Let $G=\K_{a,b}$ be the complete bipartite graph with $a+b$ vertices, let $X$ be the projective algebraic toric set parameterized by the edges of $G$, and let $X_1$ and $X_2$ be the projective tori in
$\PP^{a-1}\!\!$ and \;$\PP^{b-1}$ respectively. Then
$$
\delta_X(d) = \delta_{X_1}(d) \delta_{X_2}(d) \;.
$$
\end{theorem}

For the general case of a connected bipartite graph, the following bounds hold:

\begin{theorem}
\label{bounds mindist bipartite}
Let $G=\K_{a,b}$ be a connected bipartite graph with $a+b$ vertices, and let $X$ be the projective algebraic toric set parameterized by the edges of $G$. If $X_1$, $X_2$ and $X_3$ are the projective tori in
$\PP^{a-1}$, \;$\PP^{b-1}\!$ and \;$\PP^{a+b-2}$ respectively, then
$$
\delta_{X_1}(d) \delta_{X_2}(d) \leq \delta_X(d) \leq \delta_{X_3}(d) \;.
$$
\end{theorem}

These bounds can be explained using Lemma~\ref{bound mindist subgraph} below,
knowing that a connected bipartite graph, $G$,
contains a spanning tree and is contained in a complete bipartite graph with the same partition as $G$.

\begin{lemma}{\cite[Lemma~3.5]{mariarafael}}
\label{bound mindist subgraph}
Suppose $G$ is a subgraph of $G'$, and $X$ and $X'$ are the projective algebraic toric sets parameterized by the respective edges. If $|X| = |X'|$, then
$$
\delta_{X'}(d)  \leq \delta_X(d)  \;.
$$
\end{lemma}

\begin{example}
If $G$ is an hexagon $(n=s=6)$ and $X$ is the projective algebraic toric set parameterized by the edges of $G$, the bounds of Theorem~\ref{bounds mindist bipartite} for $q=5$ and $d=1$
are
$$
144 \leq \delta_X(1) \leq 192 \;.
$$
This is a better result than the Singleton bound, which in this case is
$$
\delta_X(1) \leq 256-6+1=251 \;.
$$
\end{example}

We end by stating a result for a connected non-bipartite graph, see {\cite{GoReSa13}}.

\begin{theorem}
Let $G$ be a connected non-bipartite graph  and let $X$ be the projective algebraic toric set parameterized by the edges of $G$. If $X'$ is the projective torus in
$\PP^{|V_G|-1}$, then
$$
\delta_{X'}(2d)  \leq \delta_X(d)  \;.
$$
\end{theorem}

 
\bibliographystyle{plain}

\begin{thebibliography}{10}


\bibitem{duursmaEtAl}
I.~M.~Duursma, C.~Renter\'ia and H.~Tapia-Recillas,
\emph{Reed-Muller codes on complete intersections}. 
Appl.~Algebra Engrg.~Comm.~Comput. 11, no.~6, 455-462 (2001).

\bibitem{eppstein}
D.~Eppstein,
\emph{Parallel recognition of series-parallel graphs}. 
Inf.~Comput.~98, No.~1, 41-55 (1992).

\bibitem{GonRen}
M. Gonz\'alez and C.~Renter\'ia,
\emph{Evaluation codes associated to complete bipartite graphs}. 
Int.~J.~Algebra 2, No.~1-4, 163-170, (2008).

\bibitem{secondhamming}
M.~González-Sarabia, C.~Rentería and M.~Hernández~de~la~Torre,
\emph{Minimum distance and second generalized Hamming weight of two
particular linear codes}.
Congr.~Numer. 161, 105-116 (2003).


\bibitem{GoReSa13} M.~Gonz\'alez, C.~Renter\'ia and E.~Sarmiento,
\emph{Parameterized codes over some embedded sets and their applications to complete graphs}.
Math.~Commun.~\textbf{18}, no.~2, 377--391 (2013).

\bibitem{neves} 
J.~Neves,
\emph{Eulerian ideals}.
ArXiv:2011.03416.

\bibitem{nested}
J.~Neves,
\emph{Regularity of the vanishing ideal over a bipartite nested ear decomposition}.
J.~lgebra Appl.~19, No.~7, Article ID 2050126, 28 p.~(2020).

\bibitem{multipartite}
J.~Neves and M.~Vaz Pinto,
\emph{Vanishing ideals over complete multipartite graphs}.
J.~Pure Appl.~Algebra 218, no.~6, 1084-1094 (2014).

\bibitem{even}
J.~Neves, M.~Vaz Pinto and R.~H.~Villarreal,
\emph{Vanishing ideals over graphs and even cycles}.
Comm.~Algebra 43, no.~3, 1050–1075 (2015).

\bibitem{joinsAndEars}
J.~Neves, M.~Vaz Pinto and R.~H.~Villarreal,
\emph{Joins, ears and Castelnuovo-Mumford regularity}.
J.~Algebra 560, 67--88 (2020).

\bibitem{parallel}
A.~Macchia, J.~Neves, M.~Vaz Pinto, R.~H.~Villarreal,
\emph{Regularity of the vanishing ideal over a parallel composition of paths}.
J.~Commut.~Algebra 12, No.~3, 391-407 (2020).

\bibitem{ReSiVi} 
C.~Renter\'ia, A.~Simis and R.~H.~Villarreal,
\emph{Algebraic methods for parameterized codes and invariants of vanishing ideals over finite fields}. 
Finite Fields Appl.~17, no.~1, 81-104 (2011).

\bibitem{torus}
E.~Sarmiento, M.~Vaz Pinto and R.~H.~Villarreal,
\emph{The minimum distance of parameterized codes on projective tori}. 
Appl.~Algebra Engrg.~Comm.~Comput.~22, no.~4, 249-264 (2011).

\bibitem{stich}
H.~Stichtenoth,
\emph{Algebraic Function Fields and Codes}.
Universitext, Springer-Verlag, Berlin, 1993.

\bibitem{tohaneanu}
S.~Tohãneanu,
\emph{Lower bounds on minimal distance of evaluation codes}.
Appl. Algebra Engrg. Comm. Comput. 20, no. 5-6, 351-360 (2009).

\bibitem{mariarafael}
M.~Vaz Pinto and R.~H.~Villarreal,
\emph{The degree and regularity of vanishing ideals of algebraic toric sets
over finite fields}. 
Comm.~Algebra 41, no.~9, 3376–3396 (2013).

\bibitem{monalg}{R. H. Villarreal, 
\emph{Monomial Algebras}.
Second edition. Monographs and Research Notes in Mathematics. CRC Press, Boca Raton, FL, 2015.}

\end{thebibliography}

\end{document}